\documentclass[leqn,10pt]{amsart}
\usepackage[utf8]{inputenc}
\usepackage{a4}
\usepackage{amsmath,amssymb,amscd,bbm}
\usepackage{enumerate}
\usepackage[mathscr]{eucal}
\usepackage[all]{xy}

\newtheorem{thm}{Theorem}[section]
\newtheorem{lemma}[thm]{Lemma}
\newtheorem{prop}[thm]{Proposition}

\theoremstyle{definition}

\newtheorem{exa}[thm]{Example}

\theoremstyle{remark}

\numberwithin{equation}{section}

\newcommand{\vanish}[1]{\relax}       
\def\qedsymbol{\hbox to 1ex{\llap{\rule{0.25pt}{1ex}}\rlap{\rule{1ex}{0.25pt}}\lower0.25pt\rlap{\raise1ex\rlap{\rule{1ex}{0.25pt}}}\hskip1ex\llap{\rule{0.25pt}{1ex}}}}
\def\rlqed{\rlap{\rule{\hsize}{0pt}\kern-1ex\kern-1em\qed}} 

\newcounter{aufzi}
\newenvironment{aufzi}{\begin{list}{ {\upshape\alph{aufzi})}}{
        \usecounter{aufzi}
        \topsep1ex
        \parsep0cm
        \itemsep0.8ex
        \leftmargin1cm
        \labelwidth0.5cm
        \labelsep0.3cm
}}
{\end{list}}

\newcounter{aufzii}
\newenvironment{aufzii}{\begin{list}{\hfill {\upshape
(\roman{aufzii})}}{
        \usecounter{aufzii}
        \topsep1ex
        \parsep0cm
        \itemsep1ex
        \leftmargin1cm
        \labelwidth0.5cm
        \labelsep0.3cm
}}
{\end{list}}

\newcounter{aufziii}
\newenvironment{aufziii}{\begin{list}{ {\upshape\arabic{aufziii})}}{
        \usecounter{aufziii}
        \topsep1ex
        \parsep0cm
        \itemsep0.8ex
        \leftmargin1cm
        \labelwidth0.5cm
        \labelsep0.3cm
}}
{\end{list}}



\newcommand{\calF}{\mathcal{F}}

\newcommand{\calQ}{\mathcal{Q}}
\newcommand{\calR}{\mathcal{R}}
\newcommand{\calS}{\mathcal{S}}

\newcommand{\calZ}{\mathcal{Z}}

\def\bfX{\mathbf{X}}

\def\uK{\mathrm{K}}
\def\uL{\mathrm{L}}

\def\uN{\mathrm{N}}

\def\ub{\mathrm{b}}

\def\ue{\mathrm{e}}

\def\ui{\mathrm{i}}

\def\uk{\mathrm{k}}
\def\ul{\mathrm{l}}

\def\un{\mathrm{n}}

\def\up{\mathrm{p}}

\def\uv{\mathrm{v}}
\def\uw{\mathrm{w}}

\renewcommand{\ue}{\mathrm{e}}    
\renewcommand{\ui}{\mathrm{i}}    
\def\id{\mathop{\mathrm{id}}\nolimits}   

\newcommand{\N}{\mathbb{N}}
\newcommand{\Z}{\mathbb{Z}}

\newcommand{\comp}[1]{#1^\mathrm{c}}    
\DeclareMathOperator{\card}{card}  
\newcommand{\sdif}{\triangle}      

\newcommand{\Bigcap}[2][\relax]{%
 \ifx#1\relax \bigcap_{#2}
 \else \bigcap^{#1}_{#2}
 \fi}
\newcommand{\Bigcup}[2][\relax]{%
 \ifx#1\relax \bigcup_{#2}
 \else \bigcup^{#1}_{#2}
 \fi}

\def\fact#1#2{#1/#2}
\def\tfact#1#2{#1/#2}

\def\fact#1#2{{\raise0.2em\hbox{$#1$}\kern-0.2em/\kern-0.1em\lower0.2em\hbox{$#2$}}}
\def\tfact#1#2{{\raise0.1em\hbox{\small$#1$}\kern-0.1em/\kern-0.1em\lower0.1em\hbox{\small$#2$}}}

\newcommand{\norm}[2][\relax]{
   \ifx#1\relax \ensuremath{\left\Vert#2\right\Vert}
   \else \ensuremath{\left\Vert#2\right\Vert_{#1}}
   \fi}

\newcommand{\Bnorm}[2][\relax]{
   \ifx#1\relax \ensuremath{\Bigl\Vert#2\Bigr\Vert}
   \else \ensuremath{\Bigl\Vert#2\Bigr\Vert_{#1}}
   \fi}

\makeatletter
\newcommand{\tdprod}[2]{\ensuremath{%
  \setbox0=\hbox{\ensuremath{\langle#1,#2 \rangle}}
  \dimen@\ht0
  \advance\dimen@ by \dp0 (#1\rule[-\dp0]{0pt}{\dimen@}\,|#2\hspace{1pt})}}
\newcommand{\dprod}[2]{\ensuremath{%
  \setbox0=\hbox{\ensuremath{\left\langle#1,#2\right\rangle}}
  \dimen@\ht0
  \advance\dimen@ by \dp0 \left\langle\left.#1\rule[-\dp0]{0pt}{\dimen@}\,\right|#2\hspace{1pt}\right\rangle}}

\newcommand{\bdprod}[2]{\ensuremath{%
  \setbox0=\hbox{\ensuremath{\bigl\langle#1,#2\bigr\rangle}}
  \dimen@\ht0
  \advance\dimen@ by \dp0 \bigl\langle#1\bigl|\rule[-\dp0]{0pt}{\dimen@}\bigr.#2\hspace{1pt}\bigr\rangle}}
\newcommand{\Bdprod}[2]{\ensuremath{%
  \setbox0=\hbox{\ensuremath{\Bigl\langle#1,#2\Bigr\rangle}}
  \dimen@\ht0
  \advance\dimen@ by \dp0 \Bigl\langle#1\Bigl|\rule[-\dp0]{0pt}{\dimen@}\Bigr.#2\hspace{1pt}\Bigr\rangle}}
\newcommand{\tsprod}[2]{\ensuremath{%
  \setbox0=\hbox{\ensuremath{(#1,#2)}}
  \dimen@\ht0
  \advance\dimen@ by \dp0 (#1\rule[-\dp0]{0pt}{\dimen@}\,|#2\hspace{1pt})}}
\newcommand{\sprod}[2]{\ensuremath{%
  \setbox0=\hbox{\ensuremath{\left(#1,#2\right)}}
  \dimen@\ht0
  \advance\dimen@ by \dp0 \left(\left.#1\rule[-\dp0]{0pt}{\dimen@}\,\right|#2\hspace{1pt}\right)}}
\newcommand{\bsprod}[2]{\ensuremath{%
  \setbox0=\hbox{\ensuremath{\bigl(#1,#2\bigr)}}
  \dimen@\ht0
  \advance\dimen@ by \dp0 \bigl(#1\bigl|\rule[-\dp0]{0pt}{\dimen@}\bigr.#2\hspace{1pt}\bigr)}}
\newcommand{\Bsprod}[2]{\ensuremath{%
  \setbox0=\hbox{\ensuremath{\Bigl(#1,#2\Bigr)}}
  \dimen@\ht0
  \advance\dimen@ by \dp0 \Bigl(#1\Bigl|\rule[-\dp0]{0pt}{\dimen@}\Bigr.#2\hspace{1pt}\Bigr)}}
\makeatother

\newcommand{\Ell}[2][\relax]{
   \ifx#1\relax \mathrm{L}^{\mathrm{#2}}
   \else \mathrm{L}^{\mathrm{#2}}_{\mathrm{#1}}
   \fi}
\renewcommand{\Ell}[2][\relax]{
   \ifx#1\relax \mathrm{L}^{\!#2}
   \else \mathrm{L}^{\!#2}_{\mathrm{#1}}
   \fi}

\newcommand{\Wee}[2][\relax]{
   \ifx#1\relax \mathrm{W}^{\mathrm{#2}}
   \else \mathrm{W}^{\mathrm{#2}}_{\mathrm{#1}}
   \fi}
\newcommand{\Har}[2][\relax]{
   \ifx#1\relax \mathsf{H}^{\mathsf{#2}}
   \else   \mathsf{H}^{\mathsf{#2}}_{\mathrm{#1}}
   \fi}

\def\prX{\mathrm X}    

\def\rlqed{\rlap{\rule{\hsize}{0pt}\kern-1ex\kern-1em\qed}}

\makeatletter

\def\maketag@@@@@#1{\llap{\hbox to\hsize{\m@th\normalfont#1}}%
\gdef\tagform@##1{\maketag@@@{(\ignorespaces##1\unskip\@@italiccorr)}}}

\def\eqtext#1{\gdef\tagform@##1{\maketag@@@@@{\ignorespaces##1\unskip\@@italiccorr\hfill}}\tag{#1}}%
\def\reqtext#1{\gdef\tagform@##1{\maketag@@@@@{\hfill\ignorespaces##1\unskip\@@italiccorr}}\tag{#1}}%
\def\leqtext#1{\gdef\tagform@##1{\maketag@@@@@{\ignorespaces##1\unskip\@@italiccorr}}\tag{#1}}%
\makeatother

\newcommand{\inv}{\mathrm{inv}}

\newcommand{\UT}[2]{\mathsf{UT}_{#1}(#2)}

\newcommand{\cntm}[1]{ \left| #1 \right|}

\newcommand{\kintl}[2]{\mathrm{int}_{#1}(#2)}
\newcommand{\kbdrl}[2]{\partial_{#1}(#2)}

\newcommand{\indi}[1]{\mathbf{1}_{#1}}

\newcommand{\kcc}[2]{\uK_{#1}^{0}(#2)}
\newcommand{\kcca}[2]{\overline{\uK}_{#1}^0(#2)}

\newcommand{\kc}[2]{\uK(#1,#2)}
\newcommand{\kca}[2]{\overline{\uK}(#1,#2)}

\newcommand{\kcs}[1]{\widetilde{\uK}(#1)}

\newcommand{\aast}{A^{\ast}}
\newcommand{\aex}{A^!}

\newcommand{\del}{\mathrm{01}}

\newcommand{\cin}[1]{\mathrm{I}(#1)}

\newcommand{\mia}[1]{m_{#1}^{\ast}}
\newcommand{\kia}[1]{k_{#1}^{\ast}}
\newcommand{\gint}[1]{G_{\mia{#1}}^{\circ}}
\newcommand{\tint}[1]{T_{#1}^{\circ}}

\date{\today}

\begin{document}

\title[Brudno Theorem]{Computable F{\o}lner monotilings\\ and a theorem of Brudno I.}

\author[Nikita Moriakov]{Nikita Moriakov}
\address{Delft Institute of Applied Mathematics, Delft University of Technology,
P.O. Box 5031, 2600 GA Delft, The Netherlands}

\email{n.moriakov@tudelft.nl}

\renewcommand{\subjclassname}{\textup{2000} Mathematics Subject
    Classification}

\thanks{The author kindly acknowledges the support from ESA CICAT of TU Delft}
\date{\today}

\begin{abstract}
The purpose of this article is to extend the earliest results of A.A. Brudno, connecting the topological entropy of a subshift $\bfX$ over $\N$ to the Kolmogorov complexity of words in $\bfX$, to subshifts over computable groups that posses computable F{\o}lner monotilings, which we introduce in this work as a computable version of the notion of a F{\o}lner monotiling originally due to B. Weiss. For every $d \in \N$, the groups $\Z^d$ and $\UT{d}{\Z}$ posses particularly nice computable F{\o}lner monotilings for which we can provide the required computing algorithms `explicitly'. Following the work of Weiss further, we show that the class of computable groups admitting computable F{\o}lner monotilings is closed under group extensions.
\end{abstract}

\maketitle

\section{Introduction}
It was proved by A.A. Brudno in \cite{brudno1974} that the topological entropy of a subshift $\bfX$ over $\N$ equals the limit of the average maximum Kolmogorov complexities of words on $[1,2,\dots,n] \subset \N$, i.e.
\begin{equation*}
h(\bfX) = \lim\limits_{n \to \infty} \frac 1 n \max\limits_{\omega \in \prX} \frac{\uK(\omega|_{[1,\dots,n]})}{n},
\end{equation*}
where $h(\bfX)$ is the topological entropy of $\bfX$ and $\uK(\omega|_{[1,\dots,n]})$ is the \emph{Kolmogorov complexity} of the word $\omega|_{[1,\dots,n]}$ of length $n$. Roughly speaking, $\uK(\omega|_{[1,\dots,n]})$ is the length of the shortest description of $\omega|_{[1,\dots,n]}$ for a fixed `optimal decompressor' that takes finite binary words as an input and produces finite words as an output. So, for instance, the Kolmogorov complexity of long periodic words will be small compared to their length, while `random' words would be most complex. The notion of a topological entropy has been since then extended to actions of discrete amenable groups, so one can wonder if the results in \cite{brudno1974} can be extended as well. In this article we give such an extension for a large class of amenable groups, including both classical examples, such as the groups $\Z^d$ for $d \in \N$, and interesting new ones, such as the nilpotent groups $\UT{d}{\Z}$ of upper-triangular matrices with integer entries, as well.

The paper is structured as follows. We provide a basic background on amenable groups and topological entropy for amenable group actions in Section \ref{ss.amengr}. Section \ref{ss.regfoln} is devoted to the notion of a F{\o}lner monotiling, which was suggested by B. Weiss in \cite{weiss2001}. The classical concepts of a computable function, a computable set and Kolmogorov complexity are explained in Section \ref{ss.compcompl}. In Section \ref{ss.compspaces} we introduce new notions: computable spaces, morphisms between computable spaces, word presheaves and Kolmogorov complexity of sections of word presheaves. Asymptotic Kolmogorov complexity of word presheaves is defined at the end of this section, too. Using this language we define computable groups in Section \ref{ss.compgrp}. `Computable version' of the notion of a F{\o}lner monotiling is developed in Section \ref{ss.compfoln}, at the end of which we prove that the class of computable groups admitting computable F{\o}lner monotilings is closed under `computable' group extensions.

The main theorem of this article is Theorem \ref{thm.brudno} in Section \ref{s.brudnocfm}, where we prove that the topological entropy of a subshift over a computable group admitting computable normal F{\o}lner monotiling equals the asymptotic Kolmogorov complexity of the associated word presheaf. The requirement of normality is nonrestrictive: we show in Lemma \ref{l.normalization} that each computable F{\o}lner monotiling gives rise to a computable normal F{\o}lner monotiling.

The groups $\Z^d$ and  $\UT{d}{\Z}$ do admit computable F{\o}lner monotilings, thus Theorem \ref{thm.brudno} gives a nontrivial abstract generalization of the results of A.A.Brudno in \cite{brudno1974} and, partially, the results of S. G. Simpson in \cite{simpson2015} for $\Z$ and $\Z^d$ cases respectively.

I would like to thank my advisor Markus Haase for reading the draft and providing corrections. I would also like to thank Stephen G. Simpson for explaining certain details in \cite{simpson2015}.

\section{Preliminaries}

\subsection{Amenable groups and entropy theory}
\label{ss.amengr}
In this section we will remind the reader of the classical notion of amenability, and state some results from entropy theory of amenable group actions. We stress that all the groups that we consider in this paper are discrete and countably infinite.

Let $\Gamma$ be a group with the counting measure $\cntm{\cdot}$. A sequence of finite sets $(F_n)_{n \geq 1}$ is called
\begin{aufziii}
\item a \textbf{weak F{\o}lner sequence} if for every finite set $K \subseteq \Gamma$ one has
\begin{equation*}
    \frac{\cntm{F_n \sdif K F_n}}{\cntm{F_n}} \to 0 ;
\end{equation*}
\item a \textbf{strong F{\o}lner sequence} if for every finite set $K \subseteq \Gamma$ one has
\begin{equation*}
    \frac{\cntm{\kbdrl{K}{F_n}}}{\cntm{F_n}} \to 0 ,
\end{equation*}
where
\begin{equation*}
\kbdrl{K}{F}:=K^{-1}F \cap K^{-1} \comp{F}
\end{equation*}
 is the \textbf{$K$-boundary} of $F$;

\end{aufziii}

One can show that a sequence of sets $(F_n)_{n \geq 1}$ is a weak F{\o}lner sequence if and only if it is a strong  F{\o}lner sequence (see \cite{ceccherini2010}, Section 5.4), hence we will simply call it a \textbf{F{\o}lner sequence}. A group $\Gamma$ is called \textbf{amenable} if it admits a F{\o}lner sequence. Since $\Gamma$ is infinite, for every F{\o}lner sequence $(F_n)_{n \geq 1}$ we have $\cntm{F_n} \to \infty$ as $n \to \infty$. For finite sets $F, K \subseteq \Gamma$ the set
\begin{equation*}
\kintl{K}{F}:=F \setminus \kbdrl{K}{F}
\end{equation*}
is called the \textbf{$K$-interior} of $F$. It is clear that if a sequence of finite sets $(F_n)_{n \geq 1}$ is a  F{\o}lner sequence, then for every finite $K \subseteq \Gamma$ one has \begin{equation*}
\cntm{\kintl{K}{F_n}}/\cntm{F_n} \to 1 \text{ as } n \to \infty.
\end{equation*}

We will now briefly remind the reader of the notion of topological entropy for amenable group actions. Let $\alpha=\{ A_1,\dots,A_n\}$ be a finite open cover of a topological space $\prX$. The \textbf{topological entropy of a cover} $\alpha$ is defined by
\begin{equation*}
H(\alpha):=\log \min\{ \card \beta : \beta \subseteq \alpha \text{ a subcover}\}.
\end{equation*}
The entropy of a cover is always a nonnegative real number. We say that a finite open cover $\alpha$ is \textbf{finer} than a finite open cover $\beta$ if for every $B \in \beta$ there exists $A \in \alpha$ such that $A \subseteq B$. If $\alpha,\beta$ are two finite open covers, then
\begin{equation*}
\alpha \vee \beta:=\{ A \cap B: A \in \alpha, B \in \beta \}
\end{equation*}
is a finite open cover as well. It is clear that $\alpha \vee \beta$ is finer than $\alpha$ and $\beta$. Given a topological dynamical system $\bfX=(\prX,\Gamma)$, where the discrete amenable group $\Gamma$ acts on the topological space $\prX$ on the left by homeomorphisms, we can also define (dynamical) entropy of a cover. For every element $g \in \Gamma$ and every finite open cover $\alpha$ we define a finite open cover $g^{-1} \alpha$ by
\begin{equation*}
g^{-1} \alpha := \{ g^{-1} A: A \in \alpha\}.
\end{equation*}
Next, for every finite subset $F \subseteq \Gamma$ and every finite open cover $\alpha$ we define a finite open cover
\begin{equation*}
\alpha^F:=\bigvee\limits_{g \in F} g^{-1} \alpha.
\end{equation*}
Let $(F_n)_{n \geq 1}$ be a F{\o}lner sequence in $\Gamma$ and $\alpha$ be a finite open cover. The limit
\begin{equation*}
h(\alpha,\Gamma):=\lim\limits_{n \to \infty} \frac{H(\alpha^{F_n})}{\cntm{F_n}}
\end{equation*}
exists and it is a nonnegative real number independent of the choice of a F{\o}lner sequence due to the lemma of D.S. Ornstein and B.Weiss (see \cite{gromov1999},\cite{krieger2007}). The limit $h(\alpha,\Gamma)$ is called the \textbf{dynamical entropy of $\alpha$}. Finally, the \textbf{topological entropy} of a topological system $\bfX=(\prX,\Gamma)$ is defined by
\begin{equation*}
h(\bfX):=\sup\{ h(\alpha,\Gamma): \alpha \text{ a finite open cover of } \prX\}.
\end{equation*}

\subsection{F{\o}lner monotilings}
\label{ss.regfoln}
The purpose of this section is to discuss notion of a F{\o}lner monotiling, that was introduced by B.Weiss in \cite{weiss2001}. The (adapted) notion of a \emph{normal} F{\o}lner monotiling, central to the results of this paper, will also be suggested below.

A \textbf{monotiling} $[F, \calZ]$ in a discrete group $\Gamma$ is a pair of a finite set $F\subseteq \Gamma$, which we call a \textbf{tile}, and a set $\calZ \subseteq \Gamma$, which we call a set of \textbf{centers}, such that $\{ F z: z \in \calZ \}$ is a covering of $\Gamma$ by disjoint translates of $F$. A \textbf{F{\o}lner monotiling} is a sequence of monotilings $([F_n, \calZ_n])_{n \geq 1}$ s.t. $(F_n)_{n \geq 1}$ is a F{\o}lner sequence in $\Gamma$. We call a F{\o}lner monotiling $([F_n,\calZ_n])_{n \geq 1}$ \textbf{normal} if
\begin{aufzi}
\item $\frac{\cntm{F_n}}{\log n} \to \infty$ as $n \to \infty$;
\item $\ue \in F_n$ for every n.
\end{aufzi}

\begin{exa}
Consider the group $\Z^d$ for some $d \geq 1$ and the F{\o}lner sequence $(F_n)_{n\geq 1}$ in $\Z^d$ given by $$F_n:= [0,1,2,\dots,n-1]^d.$$ Furthermore, for every $n$ let $$\calZ_n:= n \Z^d.$$ It is easy to see that $([F_n,\calZ_n])_{n \geq 1}$ is a normal F{\o}lner monotiling.
\end{exa}

Later we will see that the requirement of normality is not essentially restrictive for our purposes. We will need the following simple
\begin{prop}
\label{prop.fmonot}
Let $([F_n, \calZ_n])_{n\geq 1}$ be a F{\o}lner monotiling of $\Gamma$ s.t. $\ue \in F_n$ for every $n$. Then for every fixed $k$
\begin{equation}
\frac{\cntm{\kintl{F_k}{F_n} \cap \calZ_k}}{\cntm{F_n}} \to \frac 1 {\cntm{F_k}}
\end{equation}
and
\begin{equation}
\frac{\cntm{F_n \cap \calZ_k}}{\cntm{F_n}} \to \frac 1 {\cntm{F_k}}
\end{equation}
as $n \to \infty$.
\end{prop}
\begin{proof}
Observe first that for every set $A \subseteq \Gamma$  we have
\begin{equation*}
g \in \kintl{F_k}{A} \Leftrightarrow F_k g \subseteq A.
\end{equation*}
For every $n \in \N$, consider the finite set $A_{n,k}:=\{ z \in \calZ_k: F_k z \cap \kintl{F_k}{F_n} \neq \varnothing \}$. Then the translates $\{ F_k z: z \in A_{n,k} \}$ form a disjoint cover of the set $\kintl{F_k}{F_n}$. It is easy to see that
\begin{equation*}
\Gamma = \kintl{F_k}{F_n} \sqcup \kbdrl{F_k}{F_n} \sqcup \kintl{F_k}{\comp{F_n}}. 
\end{equation*}
Since $A_{n,k} \cap \kintl{F_k}{\comp{F_n}} = \varnothing$, we can decompose the set of centers $A_{n,k}$ as follows
\begin{equation*}
A_{n,k} = (A_{n,k} \cap \kintl{F_k}{F_n}) \sqcup (A_{n,k} \cap \kbdrl{F_k}{F_n}).
\end{equation*}
Since $(F_n)_{n \geq 1}$ is a F{\o}lner sequence,
\begin{equation*}
\frac{\cntm{F_k (A_{n,k} \cap \kbdrl{F_k}{F_n})}}{\cntm{F_n}}=\frac{\cntm{F_k} \cdot \cntm{A_{n,k} \cap \kbdrl{F_k}{F_n}}}{\cntm{F_n}} \to 0
\end{equation*}
and $\cntm{\kintl{F_k}{F_n}}/\cntm{F_n} \to 1$ as $n \to \infty$. Then from the inequalities
\begin{align*}
\frac{\cntm{\kintl{F_k}{F_n}}}{\cntm{F_n}} &\leq \frac{\cntm{F_k (A_{n,k} \cap \kbdrl{F_k}{F_n})}}{\cntm{F_n}} + \frac{\cntm{ F_k (A_{n,k} \cap \kintl{F_k}{F_n}) }}{\cntm{F_n}} \\
&\leq \frac{\cntm{F_k (A_{n,k} \cap \kbdrl{F_k}{F_n}) }}{\cntm{F_n}} + 1
\end{align*}
we deduce that
\begin{equation}
\frac{\cntm{F_k} \cdot \cntm{A_{n,k} \cap \kintl{F_k}{F_n}}}{\cntm{F_n}} \to 1
\end{equation}
as $n \to \infty$. It remains to note that $A_{n,k} \cap \kintl{F_k}{F_n} = \calZ_k \cap \kintl{F_k}{F_n}$ and the first statement follows. The second statement follows from the first one and the fact that $(F_n)_{n \geq 1}$ is a strong F{\o}lner sequence.
\end{proof}

Later, in the Section \ref{ss.compfoln}, we will add a \emph{computability} requirement to the notion of a normal F{\o}lner monotiling. The central result of this paper says that the Brudno's theorem holds for groups admitting a computable normal F{\o}lner monotiling.

\subsection{Computability and Kolmogorov complexity}
\label{ss.compcompl}
In this section we will discuss the standard notions of computability and Kolmogorov complexity that we will use in this work. We refer to Chapter 7 in \cite{hedman2004} for details, more definitions and proofs.

For a natural number $k$ a $k$-ary \textbf{partial function} is any function of the form $f: D \to \N \cup \{ 0 \}$, where $D$, \textbf{domain of definition}, is a subset of $(\N \cup \{ 0 \})^k$ for some natural $k$. A $k$-ary partial function is called \textbf{computable} if there exists an algorithm which takes a $k$-tuple of nonnegative integers $(a_1,a_2,\dots,a_k)$, prints $f((a_1,a_2,\dots,a_k))$ and terminates if $(a_1,a_2,\dots,a_k)$ is in the domain of $f$, while yielding no output otherwise. A function is called \textbf{total}, if it is defined everywhere.

The term \emph{algorithm} above stands, informally speaking, for a computer program. One way to formalize it is through introducing the class of \emph{recursive functions}, and the resulting notion coincides with the class of functions computable on \emph{Turing machines}. We do not focus on these question in this work, and we will think about computability in an `informal' way.

A set $A \subseteq \N$ is called \textbf{recursive} (or \textbf{computable}) if the indicator function $\indi{A}$ of $A$ is computable. It is easy to see that finite and co-finite subsets of $\N$ are computable. Furthermore, for computable sets $A,B \subseteq \N$ their union and intersection are also computable. If a total function $f: \N \to \N$ is computable and $A \subseteq \N$ is a computable set, then $f^{-1}(A)$, the full preimage of $A$, is computable. The image of a computable set via a total computable bijection is computable, and the inverse of such a bijection is a computable function.

A sequence of subsets $(F_n)_{n \geq 1}$ of $\N$ is called \textbf{computable} if the total function $\indi{F_{\cdot}}: (n,x) \mapsto \indi{F_n}(x)$ is computable. It is easy to see that a total function $f: \N \to \N$ is computable if and only if the sequence of singletons $(f(n))_{n \geq 1}$ is computable in the sense above.

It is very often important to have a numeration of elements of a set by natural numbers. A set $A \subseteq \N$ is called \textbf{enumerable} if there exist a total computable surjective function $f: \N \to A$. If the set $A$ is infinite, we can also require $f$ to be injective. This leads to an equivalent definition because an algorithm computing the function $f$  can be modified so that no repetitions occur in its output. Finite and cofinite sets are enumerable. It can be shown (Proposition 7.44 in \cite{hedman2004}) that a set $A$ is computable if and only if both $A$ and $\N \setminus A$ are enumerable. Furthermore, for a set $A \subsetneq \N$ the following are equivalent:
\begin{aufzii}
\item $A$ is enumerable;
\item $A$ is the domain of definition of a partial recursive function.
\end{aufzii}

Finally, we can introduce the Kolmogorov complexity for finite words. Let $A$ be a computable partial function defined on a domain $D$ of finite binary words with values in the set of all finite words over finite alphabet $\Lambda$. Of course, we have defined computable functions on subsets of $(\N \cup \{ 0 \})^k$ with values in $\N \cup \{ 0 \}$ above, but this can be easily extended to (co)domains of finite words over finite alphabets. We can think of $A$ as a `decompressor' that takes compressed binary descriptions (or `programs') in its domain, and decompresses them to finite words over alphabet $\Lambda$. Then we define the \textbf{Kolmogorov complexity} of a finite word $\omega$ with respect to $A$ as follows:
\begin{equation*}
\kcc{A}{\omega}:=\inf\{ l(p): A(p)=w \},
\end{equation*}
where $l(p)$ denotes the length of the description. If some word $\omega_0$ does not admit a compressed version, then we let $\kcc{A}{\omega_0} = \infty$. The \textbf{average Kolmogorov complexity} with respect to $A$ is defined by
\begin{equation*}
\kcca{A}{\omega}:=\frac{\kcc{A}{\omega}}{l(\omega)},
\end{equation*}
where $l(\omega)$ is the length of the word $\omega$. Intuitively speaking, this quantity tells how effective the compressor $A$ is when describing the word $\omega$.

Of course, some decompressors are intuitively better than some others. This is formalized by saying that $A_1$ is \textbf{not worse} than $A_2$ if there is a constant $c$ s.t. for all words $\omega$
\begin{equation}
\label{eq.optdecomp}
\kcc{A_1}{\omega} \leq \kcc{A_2}{\omega}+c.
\end{equation}
A theorem of Kolmogorov says that there exist a decompressor $A_1$ that is optimal, i.e. for every decompressor $A_2$ there is a constant $c$ s.t. for all words $\omega$ the Equation \ref{eq.optdecomp} holds.

The notion of Kolmogorov complexity can be extended to words defined on finite subsets of $\N$, and this will be essential in the following sections. More precisely, let $X \subset \N$ be a finite subset, $\imath_X: X \to \{ 1,2,\dots, \card X\} $ an increasing bijection, $\Lambda$ a finite alphabet, $A$ a decompressor and $\omega \in \Lambda^Y$ a word defined on some set $Y \supseteq X$. Then we let
\begin{equation}
\label{eq.kcnsubs}
\uK_A(\omega,X):=\kcc{A}{\omega \circ \imath_X^{-1}}.
\end{equation}
and
\begin{equation}
\label{eq.kcansubs}
\overline{\uK}_A(\omega,X):=\frac{\kcc{A}{\omega \circ \imath_X^{-1}}}{\card X}.
\end{equation}
We call $\uK_A(\omega,X)$ the \textbf{Kolmogorov complexity} of $\omega$ over $X$ with respect to $A$, and $\overline{\uK}_A(\omega,X)$ is called the \textbf{mean Kolmogorov complexity} of $\omega$ over $X$ with respect to $A$. If a decompressor $A_1$ is not worse than a decompressor $A_2$ with some constant $c$, then for all $X, \omega$ above
\begin{equation*}
\uK_{A_1}(\omega,X)\leq\uK_{A_2}(\omega,X)+c.
\end{equation*}

From now on, we will (mostly) use a fixed optimal decompressor $\aast$ and write $\uK(\omega,X)$, $\overline{\uK}(\omega,X)$ omitting explicit reference to $\aast$.

When estimating the Kolmogorov complexity of words we will often have to encode nonnegative integers using binary words. We will now fix some notation that will be used later. When $n$ is a nonnegative integer, we write
$\underline \un$ for the \textbf{binary encoding} of $n$ and $\overline \un$ for the \textbf{doubling encoding} of $n$, i.e. if $b_l b_{l-1} \dots b_0$ is the binary expansion of $n$, then $\underline \un$ is the binary word $\ub_l \ub_{l-1} \dots \ub_0$ of length $l+1$ and $\overline \un$ is the binary word $\ub_l \ub_l \ub_{l-1} \ub_{l-1} \dots \ub_0 \ub_0$ of length $2l+2$. We denote the length of the binary word $\uw$ by $l(\uw)$, and is clear that $l(\underline \un) \leq \lfloor \log n\rfloor+1$ and $l(\overline \un) \leq 2 \lfloor \log n\rfloor+2$. We write $\widehat \un$ for the encoding $\overline{l(\underline \un)} \del \underline \un$ of $n$, i.e. the encoding begins with the length of the binary word $\underline \un$ encoded using doubling encoding, then the delimiter $\del$ follows, then the word $\underline n$. It is clear that $l(\widehat \un) \leq  2 \lfloor \log(\lfloor \log n \rfloor + 1) \rfloor + \lfloor \log n \rfloor + 5$. This encoding enjoys the following property: given a binary string
\begin{equation*}
\widehat x_1 \widehat x_2 \dots \widehat x_l,
\end{equation*}
the integers $x_1,\dots,x_l$ are unambiguously restored. We will call such an encoding a \textbf{simple prefix-free encoding}.

\subsection{Computable Spaces and Sheaves}
\label{ss.compspaces}
The goal of this section is to introduce the notions of \emph{computable space}, \emph{computable function} between computable spaces and \emph{word sheaves} over computable spaces. The complexity of sections of word presheaves and asymptotic complexity of word presheaves is introduced in this section as well.

An \textbf{indexing} of a set $X$ is an injective mapping $\imath: X \to \N$ such that $\imath(X)$ is a computable subset. Given an element $x \in X$, we call $\imath(x)$ the \textbf{index} of $x$. If $i \in \imath(X)$, we denote by $x_i$ the element of $X$ having index $i$. A \textbf{computable space} is a pair $(X, \imath)$ of a set $X$ and an indexing $\imath$. Preimages of computable subsets of $\N$ under $\imath$ are called \textbf{computable subsets} of $(X,\imath)$. Each computable subset $Y \subseteq X$ can be seen as a computable space $(Y, \imath|_Y)$, where $\imath|_Y$ is the restriction of the indexing function. Of course, the set $\N$ with identity as an indexing function is a computable space, and the computable subsets of $(\N, \id)$ are precisely the computable sets of $\N$ in the sense of Section \ref{ss.compcompl}.

Let $(X_1, \imath_1), (X_2, \imath_2), \dots, (X_k, \imath_k),(Y,\imath)$ be computable spaces. A (total) function $f: X_1\times X_2 \times \dots \times X_k \to Y$ is called \textbf{computable} if the function $\widetilde f: \imath_1(X_1) \times \imath_2(X_2) \times \dots \times \imath_k(X_k) \to \imath(Y)$ determined by the condition
\begin{equation*}
\widetilde f(\imath_1(x_1),\imath_2(x_2),\dots,\imath_k(x_k)) = \imath(f(x_1,x_2,\dots,x_k))
\end{equation*}
for all $(x_1,x_2,\dots,x_k) \in X_1 \times X_2 \times \dots \times X_k$ is computable. This definition extends the standard definition of computability from Section \ref{ss.compcompl} when the computable spaces under consideration are $(\N,\id)$. A computable function $f: (X,\imath_1) \to (Y,\imath_2)$ is called a \textbf{morphism} between computable spaces. This yields the definition of the \textbf{category of computable spaces}. Let $(X, \imath_1)$, $(X, \imath_2)$ be computable spaces. The indexing functions $\imath_1$ and $\imath_2$ of $X$ are called \textbf{equivalent} if $\id: (X,\imath_1) \to (X,\imath_2)$ is an isomorphism. It is clear that the classes of computable functions and computable sets do not change if we pass to equivalent indexing functions.

Given a computable space $(X,\imath)$, we call a sequence of subsets  $(F_n)_{n \geq 1}$ of $X$ \textbf{computable} if the function $\indi{F_{\cdot}}: \N \times X \to \{ 0,1\}, (n,x) \mapsto \indi{F_n}(x)$ is computable. We will also need a special notion of computability for sequences of \emph{finite} subsets of $(X,\imath)$. A sequence of finite subsets $(F_n)_{n \geq 1}$ of $X$ is called \textbf{canonically computable} if there is an algorithm that, given $n$, prints the set $\imath(F_n)$ and halts. One way to make this more precise is by introducing the canonical index of a finite set. Given a finite set $A=\{ x_1,x_2,\dots,x_k\} \subset \N$, we call the number $\cin{A}:=\sum\limits_{i=1}^k 2^{x_i}$ the \textbf{canonical index} of $A$. Hence a sequence of finite subsets $(F_n)_{n \geq 1}$ of $X$ is canonically computable if and only if the total function $n \mapsto \cin{\imath(F_n)}$ is computable. Of course, a canonically computable sequence of finite sets is computable, but the converse is not true due to the fact that there is no effective way of determining how large a finite set with a given computable indicator function is. It is easy to see that the class of canonically computable sequences of finite sets does not change if we pass to an equivalent indexing. The proof of the following proposition is straightforward:

\begin{prop}
Let $(X,\imath)$ be a computable space. Then
\begin{aufzi}
\item If $(F_n)_{n \geq 1}, (G_n)_{n \geq 1}$ are (canonically) computable sequences of sets, then the sequences of sets $(F_n \cup G_n)_{n \geq 1}$, $(F_n \cap G_n)_{n \geq 1}$ and $(F_n \setminus G_n)_{n \geq 1}$ are (canonically) computable.
\item If $(F_n)_{n \geq 1}$ is a canonically computable sequence of sets and $(G_n)_{n \geq 1}$ is a computable sequence of sets, then the sequence of sets $(F_n \cap G_n)_{n \geq 1}$ is canonically computable.
\end{aufzi}

\end{prop}

Let $(X,\imath)$ be a computable space and $\Lambda$ be a finite alphabet. A \textbf{word presheaf} $\calF_{\Lambda}$ on $X$ consists of
\begin{aufziii}
\item A set $\calF_{\Lambda}(U)$ of $\Lambda$-valued functions defined on the set $U$ for every computable subset $U \subseteq X$;
\item A restriction mapping $\rho_{U,V}: \calF_{\Lambda}(U) \to \calF_{\Lambda}(V)$ for each pair $U,V$ of computable subsets s.t. $V \subseteq U$, that takes functions in $\calF_{\Lambda}(U)$ and restricts them to the subset $V$.
\end{aufziii}
It is easy to see that the standard `presheaf axioms' are satisfied: $\rho_{U,U}$ is identity on $\calF_{\Lambda}(U)$ for every computable $U \subseteq X$, and for every triple $V \subseteq U \subseteq W$ we have that $\rho_{W,V} = \rho_{U,V} \circ \rho_{W,U}$. Elements of $\calF_{\Lambda}(U)$ are called \textbf{sections} over $U$, or \textbf{words} over $U$. We will often write $s|_V$ for $\rho_{U,V}s$, where $s \in \calF_{\Lambda}(U)$ is a section.

We have introduced Kolmogorov complexity of words supported on subsets of $\N$ in the previous section, now we want to extend this by introducing complexity of sections. Let $(X,\imath)$ be a computable space and let $\calF_{\Lambda}$ be a word presheaf over $(X,\imath)$. Let $U \subseteq X$ be a finite set and $\omega \in \calF_{\Lambda}(U)$. Then we define the \textbf{Kolmogorov complexity} of $\omega \in \calF_{\Lambda}(U)$ by
\begin{equation}
\kc{\omega}{U}:=\kc{\omega \circ \imath^{-1}}{\imath(U)}
\end{equation}
and the \textbf{mean Kolmogorov complexity} of $\omega \in \calF_{\Lambda}(U)$ by
\begin{equation}
\kca{\omega}{U}:=\kca{\omega \circ \imath^{-1}}{\imath(U)}.
\end{equation}
The quantities on the right hand side here are defined in the Equations \ref{eq.kcnsubs} and \ref{eq.kcansubs} respectively (which are special cases of the more general definition when the computable space $X$ is $(\N,\id)$).

Let $(F_n)_{n \geq 1}$ be a sequence of finite subsets of $X$ s.t. $\card F_n \to \infty$ as $n \to \infty$. Then we define \textbf{asymptotic Kolmogorov complexity} of the word presheaf $\calF_{\Lambda}$ along the sequence $(F_n)_{n \geq 1}$ by
\begin{equation*}
\kcs{\calF_{\Lambda}}:= \limsup\limits_{n \to \infty} \max\limits_{\omega \in \calF_{\Lambda}(F_n)} \kca{\omega}{F_n}.
\end{equation*}
Dependence on the sequence is omitted in the notation, but it will always be clear from the context which sequence we take. If $A'$ is some decompressor, then it follows from the optimality of $\aast$ that
\begin{equation*}
\kcs{\calF_{\Lambda}} \leq \widetilde{\uK}_{A'}(\calF_{\Lambda}).
\end{equation*}

To simplify the notation, we adopt the following convention. We say explicitly what indexing function we use when introducing a computable space, but later, when the indexing is fixed, we shall often omit the indexing function from the notation and think about computable spaces as computable subsets of $\N$, endowed with induced ordering. Words defined on subsets of a computable space become words defined on subsets of $\N$. This will help to simplify the notation without introducing much ambiguity.

\subsection{Computable Groups}
\label{ss.compgrp}
In this section we provide the definitions of a computable group and a few related notions, connecting results from algebra with computability. This section is based on \cite{rabin1960}.

Let $\Gamma$ be a group with respect to the multiplication operation $\ast$. An indexing $\imath$ of $\Gamma$ is called \textbf{admissible} if the function $\ast: (\Gamma,\imath) \times (\Gamma,\imath) \to (\Gamma,\imath)$ is a computable function in the sense of Section \ref{ss.compspaces}. A \textbf{computable group} is a pair $(\Gamma,\imath)$ of a group $\Gamma$ and an admissible indexing $\imath$. Of course, the groups $\Z^d$ and $\UT{d}{\Z}$ possess `natural' admissible indexings. For the group $\Z$ we fix the indexing
\begin{equation*}
\imath: n \mapsto 2|n| + \indi{n \geq 0}.
\end{equation*}
It is clear that the groups $\Z^d$ for $d>1$ possess admissible indexing functions such that all coordinate projections onto $\Z$, endowed with the indexing function $\imath$ above, are computable. We leave the details to the reader.

The following lemma from \cite{rabin1960} shows that in a computable group taking the inverse is a computable operation.
\begin{lemma}
Let $(\Gamma, \imath)$ be a computable group. Then the function $\inv: (\Gamma,\imath) \to (\Gamma,\imath), g \mapsto g^{-1}$ is computable.
\end{lemma}

$(\Gamma, \imath)$ is a computable space, and we can talk about computable subsets of $(\Gamma, \imath)$. A subgroup of $\Gamma$ which is a computable subset will be called a \textbf{computable subgroup}. A homomorphism between computable groups that is computable as a map between computable spaces will be called a \textbf{computable homomorphism}. The proof of the proposition below is straightforward.
\begin{prop}
Let $(\Gamma,\imath)$ be a computable group. Then the following assertions hold:
\begin{aufziii}
\item Given a computable set $A \subseteq \Gamma$ and $g \in \Gamma$, the sets $A^{-1}, gA$ and $Ag$ are computable;

\item Given a (canonically) computable sequence $(F_n)_{n \geq 1}$ of (finite) subsets  of $\Gamma$ and $g \in \Gamma$, the sequences $(g F_n)_{n \geq 1}, (F_n g)_{n \geq 1}$ are (canonically) computable.
\end{aufziii}
\end{prop}

It is interesting to see that a computable version of the `First Isomorphism Theorem' also holds.
\begin{thm}
Let $(G,\imath)$ be a computable group and let $(H,\imath|_H)$ be a computable normal subgroup, where $\imath|_H$ is the restriction of the indexing function $\imath$ to $H$. Then there is a compatible indexing function $\imath'$ on the factor group $\fact G H$ such that the quotient map $\pi: (G,\imath) \to (\fact G H, \imath')$ is a computable homomorphism.
\end{thm}
For the proof we refer the reader to the Theorem 1 in \cite{rabin1960}.

\subsection{Computable F{\o}lner sequences and computable monotilings}
\label{ss.compfoln}
The notions of an amenable group and a F{\o}lner sequence are well-known, but, since we are working with computable groups, we need to develop their `computable' versions.

Let $(\Gamma,\imath)$ be a computable group. A F{\o}lner monotiling $([F_n, \calZ_n])_{n \geq 1}$ of $\Gamma$ is called \textbf{computable} if
\begin{aufzi}
\item $(F_n)_{n \geq 1}$ is a canonically computable sequence of finite subsets of $\Gamma$;
\item $(\calZ_n)_{n \geq 1}$ is a computable sequence of subsets of $\Gamma$.
\end{aufzi}

\begin{exa}
\label{ex.zdmonot}
Consider the group $\Z^d$ for some $d \geq 1$. Let $\imath$ be any admissible indexing such that all the coordinate projections $\Z^d \to \Z$ are computable. Then the F{\o}lner sequence $F_n = [0,1,2,\dots,n-1]^d$ is canonically computable. Furthermore, the set $\calZ_n$ of centers equals $n \Z^d$ for every $n$, hence $([F_n,\calZ_n])_{n \geq 1}$ is a computable normal F{\o}lner monotiling.
\end{exa}

In general, the procedures computing F{\o}lner monotilings might be difficult to write out explicitly and, given a F{\o}lner sequence, there might be different corresponding sequences of tilings. However, for the groups $\UT{d}{\Z}$ of upper-triangular matrices with integer entries there exist very `natural' F{\o}lner monotilings that can be easily computed. For details we refer the reader to the work \cite{golodets2002} of Ya. Golodets and S. D. Sinelshchikov.

The main results of this paper are proved for groups that admit a computable \emph{normal} F{\o}lner monotiling. In the lemma below we show that this requirement is not restrictive.
\begin{lemma}[Normalization]
\label{l.normalization}
Let $(\Gamma,\imath)$ be a computable group and $([F_n,\calZ_n])_{n \geq 1}$ be a computable F{\o}lner monotiling. Then there is a computable function $r_{\cdot}: \N \to \Gamma$ and a computable function $n_{\cdot}: \N \to \N$ such that $([F_{n_i} r_{n_i}^{-1}, r_{n_i} \calZ_{n_i}])_{i\geq 1}$ is a computable normal F{\o}lner monotiling.
\end{lemma}
\begin{proof}
Let $r_i$ be the first element of the set $F_i$ for every $i$ when we view $F_i$ as a subset of $\N$ via the indexing mapping $\imath$. Then $(F_n r_n^{-1})_{n \geq 1}$ is a F{\o}lner sequence such that $\ue \in F_n r_n^{-1}$ for every $n$, and $([F_n r_n^{-1}, r_n \calZ_n])_{n \geq 1}$ is a F{\o}lner monotiling. It is clear that we can pick the function $n_{\cdot}$ such that the growth condition is satisfied as well.
\end{proof}

The following lemma shows that we can `computably' determine which elements of canonically computable F{\o}lner sequences are `invariant enough' in computable groups.
\begin{prop}
\label{prop.indexing}
Let $(F_n)_{n \geq 1}$ be a canonically computable F{\o}lner sequence in a computable group $(\Gamma, \imath)$. Then there is a computable function $i \mapsto n_i$ such that
\begin{equation}
\label{eq.indexing}
\frac{\cntm{F_{n_i} \sdif g F_{n_i}}}{\cntm{F_{n_i}}} < \frac 1 {2 i}
\end{equation}
for all $g \in K_i$ and all $i$, where $K_i$ is the set of the first $i$ elements of $\Gamma$ with respect to the indexing $\imath$.
\end{prop}
\begin{proof}
For every $i$ we let $K_i$ be the finite set defined above. We let $n_i$ be the first index such that (\ref{eq.indexing}) holds for all $g \in K_i$ (such $n_i$ exists because $(F_n)_{n \geq 1}$ is a F{\o}lner sequence).
\end{proof}

The following lemma simplifies checking the computability of $(\calZ_n)_{n \geq 1}$.
\begin{prop}
\label{prop.monoteq}
Let $(\Gamma, \imath)$ be a computable group. Let $([F_n,\calZ_n])_{n \geq 1}$ be a F{\o}lner monotiling of $\Gamma$ such that $(F_n)_{n \geq 1}$ is a canonically computable sequence of finite sets and $\ue \in F_n$ for all $n \geq 1$. Then the following assertions are equivalent:
\begin{aufzii}
\item There is a total computable function $\phi: \N^2 \to \Gamma$ such that for every $n \in \N$
\begin{equation*}
\calZ_n = \{ \phi(n,1),\phi(n,2), \dots\}.
\end{equation*}
\item The sequence of sets $(\calZ_n)_{n \geq 1}$ is computable.
\end{aufzii}
\end{prop}
\begin{proof}
One implication is clear. For the converse, note that to prove computability of the function $\indi{\calZ_{\cdot}}$ we have to devise an algorithm that, given $n\in \N$ and $g \in \Gamma$, decides whether $g \in \calZ_n$ or not. Let $\phi: \N^2 \to \Gamma$ be the function from assertion (i). Then the following algorithm answers the question. Start with $i:=1$ and compute $\ue \phi(n,i), h_{1,n} \phi(n,i),\dots, h_{k,n} \phi(n,i)$, where $F_n=\{ \ue, h_{1,n},\dots,h_{k,n}\}$. This is possible since $(F_n)_{n \geq 1}$ is a canonically computable sequence of finite sets. If $g = \ue \phi(n,i) $, then the answer is `Yes' and we stop the program. If $g =   h_{j,n} \phi(n,i)$ for some $j$, then the answer is `No' and we stop the program. If neither is true, then we set $i:=i+1$ and go to the beginning.

Since $\Gamma =  F_n \calZ_n$, the algorithm terminates for every input.
\end{proof}

The following theorem, whose proof is essentially due to B. Weiss \cite{weiss2001}, shows that the class of groups admitting computable normal F{\o}lner monotilings is closed under group extensions.
\begin{thm}
\label{thm.weiss}
Let
\begin{equation*}
1 \to (E,\imath_E) \overset{\id}{\to} (F,\imath_F) \overset{\psi}{\to} (G,\imath_G) \to 1
\end{equation*}
be an exact sequence of computable groups such that $\id, \psi$ are computable homomorphisms. Suppose that $([E_k,\calQ_k])_{k \geq 1}, ([G_m,\calS_m])_{m \geq 1}$ are computable normal F{\o}lner monotilings of the groups $(E,\imath_E)$ and $(G,\imath_G)$ respectively. Then there is a computable normal F{\o}lner monotiling $([F_l,\calR_l])_{l \geq 1}$ in the group $(F,\imath_F)$.
\end{thm}
\begin{proof}
We begin by describing an auxiliary construction that provides us with `computable' sections of computable sets in $F$ over $G$. Let $T \subseteq F$ be a computable set. Let $\indi{T}$ be the characteristic function of $T$. Now we construct a characteristic function of a computable section $T'$ of $T$ as follows:
\begin{equation*}
\indi{T'}(n) :=  \begin{cases}
    1 & \text{if } \indi{T}(n)=1 \text{ and } ((\forall \ l < n \ \ \psi(n)\neq \psi(l)) \vee (n=\ue));\\
    0 &  \text{otherwise.}
  \end{cases}
\end{equation*}
That is, $T'$ is the set of the first members of each $E$-coset in $T$ except for the coset $\ue E$, on which we pick $\ue$ instead if $\ue \in T$. Since the functions $\indi{T}, \psi$ are computable, it is easy to see that $\indi{T'}$ is computable as well. In particular, the function $x \mapsto \psi^{-1}(x)'$ from $G$ to $F$ is computable.

Let $l \in \N$ be fixed. Let $K_l=\{f_1,f_2,\dots,f_l \} \subset F$ be the set of the first $l$ elements of $F$ with respect to the indexing $\imath_F$. We will describe an algorithm that yields a tile $F_l \subset F$ such that for all $g \in K_l$ we have
\begin{equation*}
\frac{\cntm{F_l \sdif g F_l}}{\cntm{F_l}} \leq \frac 1 l.
\end{equation*}
It is easy to see that such a sequence $(F_l)_{l \geq 1}$ is a canonically computable F{\o}lner sequence. We will use Proposition \ref{prop.monoteq} to show that the corresponding sequence of centers $(\calR_l)_{l \geq 1}$ is computable, and it will be clear from the proof why the monotiling $([F_l,\calR_l])_{l \geq 1}$ is normal as well.

Consider the finite set $\psi(K_l) \subset G$. Then the maximum $I_l \geq l$ of the indices of elements of $\psi(K_l)$ is a computable function of $l$. We let $Q_l:=\{ g_1,g_2,\dots,g_{I_l} \}$ be the finite set of the first $I_l$ elements of $G$. Let $m_{\cdot}: i \mapsto m_i$ be the computable function from the Proposition \ref{prop.indexing} applied to $(G,\imath_G)$ and $(G_m)_{m \geq 1}$, then the function $m_{\cdot}^{\ast}: l \mapsto \max(m_{I_l},2l)$ is a computable function.

Let $G_{\mia{l}}$ be the corresponding F{\o}lner tile in $G$. Consider its preimage $\psi^{-1}(G_{\mia{l}})$, which is a computable subset of $F$. We note that the sequence of sets $l \mapsto \psi^{-1} (G_{\mia{l}})$ is computable because the sequence of sets $(G_m)_{m \geq 1}$ is computable and the functions $\psi, \mia{\cdot}$ are computable. Let $T_l \subset \psi^{-1}(G_{\mia{l}})$ be the computable section of $\psi^{-1}(G_{\mia{l}})$ over $G$ as defined above, then $\psi$ is bijective as a map from $T_l$ to $G_{\mia{l}}$. Observe further that the sequence of \emph{finite} sets $l \mapsto T_l$ is \emph{canonically} computable. B. Weiss proved  that if $U$ is a sufficiently invariant monotile in $E$, then $T_l U$ is a sufficiently invariant monotile in $F$. Below we examine his construction closely.

Let $\gint{l}:=\{ t \in G_{\mia{l}}: Q_l t \subset G_{\mia{l}} \} \subset G$ be the part of $G_{\mia{l}}$ that stays within $G_{\mia{l}}$ when shifted by elements of $Q_l$. It is clear that $T_l^{\circ}:=\psi^{-1}(\gint{l})' \subset T_l$, and that the sequence of finite sets $l \mapsto \tint{l}$ is canonically computable. For all $x \in K_l$ and $t \in \tint{l}$ we deduce that
\begin{equation*}
x t = \lambda_l(x,t) \rho_l(x,t),
\end{equation*}
where $\lambda_l(x,t) \in T_l$ and $\rho_l(x,t) \in E$. The functions $\lambda_{\cdot}(\cdot,\cdot)$ and $\rho_{\cdot}(\cdot,\cdot)$ are uniquely determined by this condition and are partial computable. Consider the finite subset $$P_l:=\{ \rho_l(x,t): x \in K_l, t\in \tint{l} \}.$$ The maximum index $J_l$ of elements of this set is a computable function of $l$. Let $k_{\cdot}: i\mapsto k_i$ be the computable function from the Proposition \ref{prop.indexing} applied to $(E,\imath_E)$ and $(E_k)_{k \geq 1}$, then the function $\kia{\cdot}: l \mapsto \max(k_{J_l}, \mia{l})$ is computable. Consider the F{\o}lner tile $E_{\kia{l}}$, and let $E_{\kia{l}}^{\circ}:=\{ s \in E_{\kia{l}}: P_l s \subset E_{\kia{l}}\}$ be the part of $E_{\kia{l}}$ that stays in $E_{\kia{l}}$ when shifted by elements of $P_l$. We claim that the tile $$F_l:=T_l E_{\kia{l}}$$ is `invariant enough'. Observe that $$K_l \tint{l} E_{\kia{l}}^{\circ} \subset T_l E_{\kia{l}}.$$
By definition, the set $\gint{l}$ is large enough:
\begin{equation*}
\cntm{\gint{l}} \geq \left(1-\frac{1}{2 l}\right)\cntm{G_{\mia{l}}},
\end{equation*}
hence
\begin{equation*}
\cntm{\tint{l}} \geq \left(1-\frac{1}{2 l}\right)\cntm{T_l}.
\end{equation*}
Similarly,
\begin{equation*}
\cntm{E_{\kia{l}}^{\circ}} \geq \left(1-\frac{1}{2 l}\right)\cntm{E_{\kia{l}}},
\end{equation*}
and it follows that
\begin{equation*}
\cntm{T_l^{\circ} E_{\kia{l}}^{\circ}} \geq \left(1-\frac 1 {2 l}\right)\cntm{ T_l E_{\kia{l}}}.
\end{equation*}
We deduce that for every $g \in K_l$
\begin{align*}
\frac{\cntm{F_l \sdif g F_l}}{\cntm{F_l}} = \frac{\cntm{F_l \sdif g^{-1} F_l}}{\cntm{F_l}} \leq \frac{1}{l},
\end{align*}
and this shows that $F_l$ is `invariant enough'. We have obtained a canonically computable F{\o}lner sequence $l \mapsto F_l$.

It remains to prove that for each $l$ the set $F_l$ is a tile and that the sequence of centers $(\calR_l)_{l \geq 1}$ is computable. Let $\phi_E, \phi_G$ be computable functions from the Proposition \ref{prop.monoteq} applied to groups $(E,\imath_E),(G,\imath_G)$ respectively. Let $\theta: \N^2 \to F$ be the total computable function $(n,i) \mapsto \psi^{-1}(\phi_G(n,i))'$, i.e. we compute $\phi_G(n,i)$ first and then pick an element in its fiber in $F$. It is clear that
\begin{equation*}
\{ \phi_E(\kia{l},i) \theta(\mia{l},j )\}_{i,j \geq 1}
\end{equation*}
is a set of centers for the tile $F_l$. If $\nu: \N \to \N^2$ is any computable bijection, then the total computable function $\phi_F: (n,i) \mapsto \phi_E(\kia{n}, \nu_1(i)) \theta(\mia{n}, \nu_2(i))$ satisfies the conditions of Proposition \ref{prop.monoteq}, and the proof is complete.

\end{proof}

\section{Brudno's theorem}
\label{s.brudnocfm}
We are now ready to prove the main theorem of this article. First, we will explain the definitions and introduce some notation that will be used in the proofs.

By a \textbf{subshift} $\bfX=(\prX,\Gamma)$ we mean a closed $\Gamma$-invariant subset $\prX$ of $\Lambda^{\Gamma}$, where $\Lambda$ is the finite \textbf{alphabet} of $\bfX$. The words consisting of letters from the alphabet $\Lambda$ will be often called \textbf{$\Lambda$-words}. Of course, we can assume without loss of generality that $\Lambda = \{ 1,2,\dots,k\}$ for some $k$. The left action of the group $\Gamma$ on $\prX$ is given by
\begin{equation*}
(g\cdot \omega)(x) := \omega(x g) \ \ \ \forall x,g \in \Gamma, \omega \in \prX.
\end{equation*}
We can associate a word presheaf $\calF_{\Lambda}$ to the subshift $\bfX$ by setting
\begin{equation}
\calF_{\Lambda}(F):=\{ \omega|_F: \omega \in \prX\}.
\end{equation}
That is, $\calF_{\Lambda}(F)$ is the set of all restrictions of words in $\prX$ to the set $F$ for every computable $F$.

The goal of this section is to prove the following:
\begin{thm}
\label{thm.brudno}
Let $(\Gamma,\imath)$ be a computable group with a fixed computable normal F{\o}lner monotiling $([F_n,\calZ_n])_{n \geq 1}$. Let $\bfX=(\prX,\Gamma)$ be a subshift on $\Gamma$ and $\calF_{\Lambda}$ be the associated word presheaf on $\Gamma$. Then
\begin{equation*}
\kcs{\calF_{\Lambda}} = h(\bfX),
\end{equation*}
where the asymptotic complexity of the word presheaf $\calF_{\Lambda}$ is computed along the sequence $(F_n)_{n \geq 1}$.
\end{thm}

The proof is split into two parts, establishing respective inequalities in Theorems \ref{thm.brudnogeq} and \ref{thm.brudnoleq}. Given a subshift $\bfX$ on the alphabet $\Lambda$, we define the cover
\begin{equation*}
\alpha_{\Lambda}:=\{A_1,\dots,A_k\}, \ A_i:=\{ \omega \in \prX: \omega(\ue) = i\} \text{ for } i = 1,\dots,k.
\end{equation*}
Then $\alpha_{\Lambda}$ is, clearly, a generating cover: for every finite open cover $\beta$ of $\prX$ there exists a finite subset $F \subseteq \Gamma$ such that the cover $\alpha_{\Lambda}^F$ is finer than $\beta$. We will use the following well-known
\begin{prop}
Let $\bfX$ be a subshift of $\Lambda^{\Gamma}$ and $\alpha_{\Lambda}$ be the cover defined above.  Then
\begin{equation*}
h(\alpha_{\Lambda},\Gamma) = \lim\limits_{n\to \infty} \frac{\log \card \calF_{\Lambda}(F_n)}{\cntm{F_n}}= h(\bfX).
\end{equation*}
\end{prop}

We will now establish the first inequality. The proof is essentially the same as the original one from \cite{brudno1974}.
\begin{thm}
\label{thm.brudnoleq}
In the setting of Theorem \ref{thm.brudno} we have
\begin{equation}
h(\bfX) \leq \kcs{\calF_{\Lambda}}
\end{equation}
\end{thm}
\begin{proof}
By the definition, we have to show that
\begin{equation*}
\lim\limits_{n\to \infty} \frac{\log \card \calF_{\Lambda}(F_n)}{\cntm{F_n}} \leq \limsup\limits_{n \to \infty} \max\limits_{\omega \in \calF_{\Lambda}(F_n)} \kca{\omega}{F_n}.
\end{equation*}
Suppose that $\kcs{\calF_{\Lambda}} < t$ for some $t \geq 0$. Then there exists $n_0$ such that for all $n \geq n_0$ and all $\omega \in \calF_{\Lambda}(F_n)$ the inequality $\kc{\omega}{F_n} \leq t \cntm{F_n}$ holds. There are at most $2^{t \cntm{F_n} + 1}$ valid binary programs for the decompressor $\aast$ of length at most $t \cntm{F_n}$, hence $\card \calF_{\Lambda}(F_n) \leq 2^{t \cntm{F_n}+1}$. Taking the logarithm shows that for all $n \geq n_0$ we have
\begin{equation*}
\frac{\log \card \calF_{\Lambda}(F_n)}{\cntm{F_n}} \leq t+ \frac{1}{|F_n|},
\end{equation*}
and this completes the proof.
\end{proof}

The proof of the second inequality requires more work. The proof we provide is based on the idea of the proof of Lemma 5.1 from \cite{simpson2015}.
\begin{thm}
\label{thm.brudnogeq}
In the setting of Theorem \ref{thm.brudno} we have
\begin{equation}
h(\bfX) \geq \kcs{\calF_{\Lambda}}
\end{equation}
\end{thm}
\begin{proof}
By the definition, we have to show that
\begin{equation*}
\limsup\limits_{n \to \infty} \max\limits_{\omega \in \calF_{\Lambda}(F_n)} \kca{\omega}{F_n} \leq \lim\limits_{n\to \infty} \frac{\log \card \calF_{\Lambda}(F_n)}{\cntm{F_n}}.
\end{equation*}

The alphabet $\Lambda$ is finite, so we encode each letter of $\Lambda$ using precisely $\lfloor \log \card \Lambda \rfloor+1$ bits. We fix this encoding. Then binary words of length $N\left(\lfloor \log \card \Lambda \rfloor+1 \right)$ are unambiguously interpreted as $\Lambda$-words of length $N$. We will now describe a decompressor $\aex$ that will be used to prove the theorem. The decompressor is defined on the domain of binary programs of the form
\begin{equation}
\label{eq.prog}
\up = \widehat \uk \widehat \un \widehat \uN \widehat \uL  \widehat \ul \uw_1 \uw_2 \dots \uw_N \uv \widehat{\ui}_1 \widehat{\ui}_2 \dots \widehat{\ui}_s.
\end{equation}
Here $\widehat \uk, \widehat \un, \widehat \uN, \widehat \uL, \widehat \ul$ are simple prefix-free encodings of the natural numbers $k,n,N,L,l$. Binary words $\uw_1, \uw_2, \dots, \uw_N$ have all length $L$. Words $\widehat{\ui}_1, \widehat{\ui}_2, \dots, \widehat{\ui}_s$ encode some natural numbers $i_1,i_2,\dots,i_s$ that are required to be less or equal to $N$. Finally, $\uv$ is a binary word of length $l$. Observe that programs of the form \ref{eq.prog} are indeed unambiguously interpreted.

The decompressor $\aex$ works as follows. First, given $k,n$ above, the finite sets $$F_k, F_n, \kintl{F_k}{F_n}, \kintl{F_k}{F_n} \cap \calZ_k \subseteq \N$$ are computed. We let $I_{k,n}:=\kintl{F_k}{F_n} \cap \calZ_k$ and compute the set $$\Delta_{k,n}:=F_k \left( \kintl{F_k}{F_n} \cap \calZ_k \right) \subseteq F_n.$$ We treat $N$ binary words $\uw_1, \uw_2, \dots, \uw_N$ of length $L$ as encodings of $\Lambda$-words $w_1,w_2,\dots,w_N$ of length $F_k$, if $L \neq \cntm{F_k} \left( \lfloor \log \card \Lambda \rfloor+1 \right)$ the algorithm terminates without producing output. Words $w_1,w_2,\dots,w_N$ form the \emph{dictionary} that we will use to encode parts of the words. We require that $s = \cntm{\kintl{F_k}{F_n} \cap \calZ_k}$ and $$l = \cntm{F_n \setminus \Delta_{k,n}} \left( \lfloor \log \card \Lambda \rfloor+1 \right),$$ the algorithm terminates without producing output if this does not hold. Otherwise, the binary word $\uv$ of length $l$ is seen as a binary encoding of the $\Lambda$-word $v$ of length $\cntm{F_n \setminus \Delta_{k,n}}$.

We will now compute a $\Lambda$-word $\omega$ defined on $F_n$. The set $\kintl{F_k}{F_n} \cap \calZ_k$ is ordered as a subset of $\N$. For $j$-th element $g_j \in \kintl{F_k}{F_n} \cap \calZ_k$ we require that $\omega|_{F_k g_j} \circ \imath_{F_k g_j}^{-1} = w_{i_j}$, where $j=1,2,\dots,s$. That is, we require that the restriction of $\omega$ to the subset $F_k g_j$ coincides with $i_j$-th element of the dictionary for every $j$. It is clear that this determines the restriction $\omega|_{\Delta_{k,n}}$, and it remains to describe $\omega|_{F_n \setminus \Delta_{k,n}}$. We require that $\omega|_{F_n \setminus \Delta_{k,n}} \circ \imath_{F_n \setminus \Delta_{k,n}}^{-1} = v$. The decompressor $\aex$ prints the $\Lambda$-word $\omega \circ \imath_{F_n}^{-1}$.

Fix $k \geq 1$ and $\varepsilon>0$. Let $n_0$ be such that for all $n \geq n_0$ we have
\begin{equation*}
\frac{\cntm{F_n \setminus \Delta_{k,n}}}{\cntm{F_n}} \leq \varepsilon.
\end{equation*}
Let $\omega \in \calF_{\Lambda}(F_n)$. We use the following program to encode $\omega$. We let $N := \card \calF_{\Lambda}(F_k)$, $L:=\cntm{F_k} \left( \lfloor \log \card \Lambda \rfloor+1 \right)$ and $w_1,w_2,\dots,w_N$ be the list of words $\upsilon \circ \imath_{F_k}^{-1}$ for $\upsilon \in \calF_{\Lambda}(F_k)$ (say, in lexicographic order). For every $g_j \in \kintl{F_k}{F_n} \cap \calZ_k$ and every $x \in F_k$ note that
\begin{equation*}
\omega|_{F_k g_j} (x g_j) = (g_j \cdot \omega)|_{F_k}(x),
\end{equation*}
where $g_j \cdot \omega \in \prX$ by invariance. Hence we let $i_j$ be the index of the word $(g_j \cdot \omega)|_{F_k} \circ \imath_{F_k}^{-1}$ in the dictionary $w_1,w_2,\dots,w_N$ for every $j=1,2,\dots, \cntm{\kintl{F_k}{F_n} \cap \calZ_k}$. Finally, we let $v$ be the remainder $\omega|_{F_n \setminus \Delta_{k,n}} \circ \imath_{F_n \setminus \Delta_{k,n}}^{-1}$ and $l$ be the length of the binary encoding of the word $v$. It is clear that the program \ref{eq.prog} with the parameters determined above does describe $\omega|_{F_n}$.

It remains to estimate the length of $\up$. It is easy to see that
\begin{align*}
l(\up) &\leq l(\widehat \uk) + l(\widehat \un)+l(\widehat{\uN})+ l(\widehat{\uL})+l(\widehat{\ul})+\card{\calF_{\Lambda}(F_k)} \cntm{F_k} \left( \lfloor \log \card \Lambda \rfloor+1 \right)+\\
&+\cntm{F_n \setminus \Delta_{k,n}}\left( \lfloor \log \card \Lambda \rfloor+1 \right)+\cntm{I_{k,n}}l(\widehat{\uN}),
\end{align*}
and taking the limit as $n \to \infty$ we see (using Proposition \ref{prop.fmonot}) that
\begin{align*}
\limsup\limits_{n \to \infty} & \max\limits_{\omega \in \calF_{\Lambda}(F_n)} \kca{\omega}{F_n} \leq \varepsilon \left( \lfloor \log \card \Lambda \rfloor+1 \right)+ \\
&+\frac{2 \lfloor \log(\lfloor \log \card \calF_{\Lambda}(F_k) \rfloor + 1) \rfloor + \lfloor \log \card \calF_{\Lambda}(F_k) \rfloor + 5}{\cntm{F_k}}.
\end{align*}
Since $k, \varepsilon$ are arbitrary the conclusion follows.

\end{proof}

It is clear that the proof of Theorem \ref{thm.brudno} now follows from Theorem \ref{thm.brudnogeq} and Theorem \ref{thm.brudnoleq}.

\bibliographystyle{acm}
\bibliography{zdbrudnobib}
\end{document}